\documentclass[a4paper,11pt]{amsart}
\usepackage[top=1.25in, bottom=1.25in, left=1.25in, right=1.25in]{geometry}
\usepackage{graphicx}
\usepackage{amsfonts}
\usepackage{epsf}
\usepackage{amssymb}
\usepackage{amsmath}
\usepackage{amscd}
\usepackage{amsthm}
\usepackage{tikz}
\usetikzlibrary{
  cd,
  calc,
  positioning,
  arrows,
  decorations.pathreplacing,
  decorations.markings,
}
\usepackage{pdfpages}
\usepackage{fancyhdr}
\usepackage{setspace}
\usepackage{hyperref}
\usepackage[all]{xy}
\usetikzlibrary{matrix}
\usepackage{verbatim}
\usepackage{enumerate}
\usepackage{mathrsfs}
\usepackage{pinlabel}
\usepackage{lscape}
\usepackage{color}
\usepackage{enumitem}

\newtheorem{theorem}{Theorem}[section]

\newtheorem{proposition}[theorem]{Proposition}
\newtheorem{corollary}[theorem]{Corollary}
\newtheorem{lemma}[theorem]{Lemma}

\newtheorem{question}[theorem]{Question}

\theoremstyle{definition}
\newtheorem{definition}[theorem]{Definition}

\newtheorem*{case2'}{Case 2$'$}
\newtheorem{theorem-named}{}
\newtheorem{theorem-labeled}{Theorem}
\newtheorem{definition-named}{}
\newtheorem{conjecture-named}{}
\newtheorem{case-named}{}

\numberwithin{equation}{section}

\newcommand{\tors}{{\rm Tors}}

\newcommand{\Z}{\mathbb{Z}}
\newcommand{\Q}{\mathbb{Q}}
\newcommand{\Qp}{\mathbb{Q}_{>0}}

\newcommand{\lp}{L(p,q)}

\newcommand{\qhom}{{\Theta}_{\Q}^{3}}
\newcommand{\rhom}{{\Theta}_{R}^{3}}

\newcommand{\lens}{\mathcal{L}}

\newcommand{\s}{\mathfrak{s}}

\def\Z{\mathbb{Z}}

\def\Q{\mathbb{Q}}

\newcommand{\Id}{\operatorname{Id}}

\newcommand{\rk}{\operatorname{rk}}

\DeclareMathOperator\Coker{Coker}

\makeatletter

\begin{document}
\title{Rational cobordisms and integral homology}

\author{Paolo Aceto}
\address{Mathematical Institute University of Oxford, Oxford, United Kingdom}
\email{paoloaceto@gmail.com}
\urladdr{www.maths.ox.ac.uk/people/paolo.aceto}
\author{Daniele Celoria}
\address{Mathematical Institute University of Oxford, Oxford, United Kingdom}
\email{daniele.celoria@maths.ox.ac.uk}
\urladdr{www.maths.ox.ac.uk/people/daniele.celoria}
\author{JungHwan Park}
\address{School of Mathematics, Georgia Institute of Technology, Atlanta, GA, USA}
\email{junghwan.park@math.gatech.edu }
\urladdr{people.math.gatech.edu/~jpark929/}
\def\subjclassname{\textup{2010} Mathematics Subject Classification}
\expandafter\let\csname subjclassname@1991\endcsname=\subjclassname
\expandafter\let\csname subjclassname@2000\endcsname=\subjclassname
\subjclass{%
  57N13, 
  57M27, 
  57N70, 
  57M25
}

\keywords{Lens spaces, 3-dimensional homology cobordism groups}

\begin{abstract}
We consider the question of when a rational homology $3$-sphere is rational homology cobordant to a connected sum of lens spaces. We prove that every rational homology cobordism class in the subgroup generated by lens spaces is represented by a unique connected sum of lens spaces whose first homology group injects in the first homology group of any other element in the same class. As a first consequence, we show that several natural maps to the rational homology cobordism group have infinite rank cokernels. Further consequences include a divisibility condition between the determinants of a connected sum of $2$-bridge knots and any other knot in the same concordance class. Lastly, we use knot Floer homology combined with our main result to obstruct Dehn surgeries on knots from being rationally cobordant to lens spaces.
\end{abstract}

\maketitle

\section{Introduction}\label{sec:intro}
For any abelian group $R$, a smooth, closed, oriented, and connected $3$-manifold $Y$ such that $H_*(Y; R) \cong H_*(S^3; R)$ is called a \emph{$R$-homology sphere}. A smooth, compact, oriented, and connected $4$-manifold $X$ such that $H_*(X; R) \cong H_*(B^4; R)$ is called a \emph{$R$-homology ball}. The equivalence relation given by smooth $R$-homology cobordism on the set of $R$-homology spheres produces a group structure induced by connected sum on the equivalence classes. This group, denoted by $\rhom$, is the \emph{$3$-dimensional $R$-homology cobordism group}. Note that a $R$-homology sphere represents the trivial class if and only if it bounds an $R$-homology ball. We denote by $\lens$ the subgroup of $\qhom$ generated by lens spaces. 

Our goal is to give constraints for $\Q$-homology spheres to be contained in $\lens$. These constraints, together with results on the structure of $\lens$ from \cite{Lisca:2007-1,Lisca:2007-2}, lead to various consequences on the structure of $\qhom$ and its relation with $ \Theta_{\mathbb{Z}}^3$ and $ \Theta_{\mathbb{Z}_p}^3$ for $p$ prime.

\begin{theorem}\label{thm:main}
Any class in $\mathcal{L}$ contains a connected sum of lens spaces $L$ such that if $Y$ is $\Q$-homology cobordant to $L$, then there is an injection $$H_1(L; \mathbb{Z}) \hookrightarrow H_1(Y;\mathbb{Z}).$$ Moreover, as a connected sum of lens spaces $L$ is uniquely determined up to orientation preserving diffeomorphism.
\end{theorem}

%

We now provide a quick overview of
the general strategy for the proof. First, given any class in $\lens$, we give an algorithm to find the unique representative $L$ identified in Theorem~\ref{thm:main} (see Definition~\ref{def:reduced} and Theorem~\ref{thm:main'}). Both $L$ and $-L$ bound canonical negative definite plumbed 4-manifolds $P$ and $P^*$. If $Y$ is $\Q$-homology cobordant to $L$ via a cobordism $W$, then the union
$$P \cup_L W \cup_Y -W \cup_{-L} \cup P^*$$
is a smooth closed 4-manifold with negative definite intersection form to which Donaldson's theorem applies~\cite{Donaldson:1987-1}. By using Lisca's work~\cite{Lisca:2007-1, Lisca:2007-2}, we show that the embedding of the integral lattice  $(H_2(P;\mathbb{Z}),Q_P) \oplus (H_2(P^*;\mathbb{Z}),Q_{P^*})$ to the integral lattice with the standard negative definite form is unique, and the two summands are each other's respective orthogonal complement (see Proposition~\ref{prop:isomorphic}). The result that $H_1(L; \mathbb{Z}) \hookrightarrow H_1(Y;\mathbb{Z})$ then follows easily.


Note that $\Z$-homology spheres and $\Z_p$-homology spheres are $\Q$-homology spheres, and $\Z$-homology balls and $\Z_p$-homology balls are also $\Q$-homology balls. Hence there are natural maps $\psi \colon \Theta_\mathbb{Z}^3 \rightarrow \Theta_\mathbb{Q}^3$ and $\psi_p \colon \Theta_{\mathbb{Z}_p}^3 \rightarrow \Theta_\mathbb{Q}^3$. It is an interesting problem to understand the properties of these maps. It was first shown by S.\ Kim and Livingston \cite{Kim-Livingston:2014-1} following the work of Hedden, Livingston, and Ruberman \cite{Hedden-Livingston-Ruberman:2012-1} that $\Coker \psi$ contains a subgroup isomorphic to $\mathbb{Z}^\infty \oplus \mathbb{Z}_2^\infty$ (see also \cite{Hedden-Kim-Livingston:2016-1,Golla-Larson:2018-1}). In fact, examples from \cite{Hedden-Livingston-Ruberman:2012-1, Kim-Livingston:2014-1, Hedden-Kim-Livingston:2016-1} bound topological $\Q$-homology balls. The first author and Larson in \cite{Aceto-Larson:2017-1} showed that the intersection of the image of $\psi$ and $\mathcal{L}$ is trivial, which immediately implies that $\Coker \psi$ contains a subgroup isomorphic to $\mathbb{Z}^\infty \oplus \mathbb{Z}_2^\infty$. We recover this fact as an easy corollary of Theorem~\ref{thm:main}.

\begin{corollary}[\cite{Aceto-Larson:2017-1}]\label{cor:AK} Let $\psi \colon \Theta_\mathbb{Z}^3 \rightarrow \Theta_\mathbb{Q}^3$ be the homomorphism induced by the inclusion. Then $$\psi(\Theta_\mathbb{Z}^3) \cap \mathcal{L} = 0.$$ In particular, $\Coker \psi$ contains a subgroup isomorphic to $\mathbb{Z}^\infty \oplus \mathbb{Z}_2^\infty$.
\end{corollary}

We also have a characterization of the intersection of the image of $\psi_p$ and $\mathcal{L}$ for any prime $p$. Note that the subgroup generated by lens spaces $L(r,s)$ with $\gcd(r,p)=1$ is contained in the intersection of $\psi_p(\Theta_{\mathbb{Z}_p}^3)$ and $\mathcal{L}$. 

\begin{corollary}\label{cor:Zp} For any prime $p$, let $\psi_p \colon \Theta_{\mathbb{Z}_p}^3 \rightarrow \Theta_\mathbb{Q}^3$ be the homomorphism induced by the inclusion. Then $$\psi_p(\Theta_{\mathbb{Z}_p}^3) \cap \mathcal{L} = \langle\{ L(r,s) \mid \gcd(r,p)=1 \}\rangle.$$ 
As a consequence, $\Coker \psi_p$ contains a subgroup isomorphic to $\mathbb{Z}^\infty$ if $p \equiv 3 \pmod{4}$ and $\mathbb{Z}^\infty \oplus \mathbb{Z}_2^\infty$ otherwise.\end{corollary}

We remark that on a similar note S.\ Kim and Livingston \cite{Kim-Livingston:2014-1} showed that the cokernel of $\Phi \colon \bigoplus  \Theta_{\mathbb{Z}[\frac{1}{p}]}^3\rightarrow \Theta_\mathbb{Q}^3$ contains a subgroup isomorphic to $\mathbb{Z}^\infty \oplus \mathbb{Z}_2^\infty$.

Let $\mathcal{C}$ be the smooth knot concordance group. For any prime $p$ and positive integer $r$, the $p^r$-fold cyclic branched cover of a knot $K$, denoted by $\Sigma_{p^r}(K)$, is a $\Z_p$-homology sphere. If $K$ is smoothly slice then $\Sigma_{p^r}(K)$ bounds a $\Z_p$-homology ball~\cite{Casson-Gordon:1978-1, Casson-Gordon:1986-1}. Moreover, it is easy to see that the $p^r$-fold cyclic branched cover of the connected sum of two knots is the same as the connected sum of their $p^r$-fold cyclic branched covers. Hence we get a homomorphism $\beta_{p^r}\colon \mathcal{C} \rightarrow \Theta_\mathbb{Q}^3$ defined by taking the $p^r$-fold cyclic branched cover. In fact, $\beta_{p^r}$ factors through $\psi_p$ as follows. 

$$\begin{tikzcd}
\mathcal{C}\arrow[rr, "\beta_{p^r}"] \arrow[dr]& & \Theta_\mathbb{Q}^3 \\
& \Theta_{\Z_p}^3 \arrow[ur, "\psi_p"]&
\end{tikzcd}$$
Since the image of $\beta_{p^r}$ is contained in the image of $\psi_p$, we get the following immediate corollary. Note that the kernel of $\beta_{p^r}$ was also studied in \cite{Aceto-Larson:2017-1} (see also \cite{Casson-Harer:1981-1}).

\begin{corollary}\label{cor:branch} For any prime $p$ and positive integer $r$, $\Coker \beta_{p^r}$ contains a subgroup isomorphic to $\mathbb{Z}^\infty$ if $p \equiv 3 \pmod{4}$ and $\mathbb{Z}^\infty \oplus \mathbb{Z}_2^\infty$ otherwise. \qed
\end{corollary}

Note that by considering the linking form of $\Q$-homology spheres we get a homomorphism $\qhom \rightarrow W(\Q/\Z)$, where $W(\Q/\Z)$ is the Witt group of
nonsingular $\Q/\Z$-valued linking forms on finite abelian groups. Let $\mathcal{K}$ be the kernel of this map. Conjecturally every element in $\mathcal{K}$ bounds a topological $\Q$-homology ball (see \cite{Kim-Livingston:2014-1}). It is well known that $W(\Q/\Z)$ is isomorphic to $\Z_2^\infty \oplus \Z_4^\infty$. In particular, each infinite rank subgroup we found in Corollaries~\ref{cor:AK}, \ref{cor:Zp}, \ref{cor:branch} can be taken to be a subgroup of $\mathcal{K}$.\\

In~\cite{Kim-Livingston:2014-1}, it is shown that for any square-free and relatively prime positive integers $p$ and $q$, there is no $\Q$-homology cobordism from
$L(pq, 1)$ to a connected sum $Y_1 \# Y_2$, where $H_1(Y_1;\Z) = \Z_p$ and $H_1(Y_2;\Z) = \Z_q$. Using Theorem~\ref{thm:main}, we can show that the same conclusion holds with a different assumption (see Proposition~\ref{prop:nonsplit} for a more general statement).
\begin{corollary}\label{cor:nonsplitta}
 If $\gcd(a,b) >1$ and $ab \neq 4$, the lens space $L(ab,1)$ is not $\Q$-homology cobordant to a connected sum $Y_a \# Y_b$, where $|H_1(Y_a;\Z)| = a$ and $|H_1(Y_b;\Z)| = b$.
\end{corollary}

More generally, we can consider a filtration of $\qhom$ and study the quotient of each stage. The proof of the next result follows almost directly from Theorem~\ref{thm:main} and \cite{Lisca:2007-2}.

\begin{corollary}\label{cor:filtration}Let $\mathcal{O}_n$ be the subgroup of $\qhom$ defined as follows\textup{:}
$$\mathcal{O}_n:=\langle\{ Y \mid H_1(Y;\mathbb{Z}) \text{ does not have any element of order } p^{n+1} \text{ for each prime } p \}\rangle.$$
Then $\mathcal{O}_{n+1} / \mathcal{O}_{n}$ contains a subgroup isomorphic to $\mathbb{Z}^\infty \oplus \mathbb{Z}_2^\infty$ for each positive integer $n$.
\end{corollary}

Recall that the branched double cover of a $2$-bridge knot is a lens space. Moreover, the determinant of a knot is the order of the first integral homology group of its branched double cover. The following corollary is an analogue of Theorem~\ref{thm:main} in the context of knot concordance, stating that suitable connected sums of $2$-bridge knots minimize the determinant in their concordance classes. 

\begin{corollary}\label{cor:2bridge} Any smooth concordance class in the subgroup generated by $2$-bridge knots is represented by a connected sum of $2$-bridge knots $K$ such that if $J$ is concordant to $K$, then $\det(K)$ divides $\det(J)$. Moreover, as a connected sum of $2$-bridge knots $K$ is uniquely determined up to isotopy.
\end{corollary}

The first author and Alfieri \cite{Aceto-Alfieri:2017-1} considered the problem of when a connected sum of two torus knots is concordant to an alternating knot. They show that if $T_{p,q} \# -T_{p',q'}$ is concordant to an alternating knot, then either $T_{p,q}$ and $T_{p',q'}$ are alternating knots, or their difference is of the form $T_{3,6n+1} \# -T_{3,6n+2}$. Using Corollary~\ref{cor:2bridge}, we can provide a complete answer if we restrict ourselves to $2$-bridge knots (recall that every $2$-bridge knot is alternating).

\begin{corollary}\label{cor:torus} Let $T_{p,q}$ and $T_{p',q'}$ be two distinct torus knots. Then $T_{p,q} \# -T_{p',q'}$ is concordant to a connected sum of $2$-bridge knots if and only if $T_{p,q}$ and $T_{p',q'}$ are $2$-bridge knots.
\end{corollary}

In a different direction, knot Floer homology  \cite{ozsvath2004holomorphicknot, Rasmussen:2003-1} has been a useful tool in obstructing manifolds that are obtained from Dehn surgery on knots from being lens spaces \cite{Rasmussen:2004-1,Rasmussen:2007-1, Baker-Grigsby-Hedden:2008-1, Greene:2013-1, Greene:2015-1}. As a consequence of Theorem~\ref{thm:main} we show a similar result for the $\Q$-homology cobordism class of such manifolds. Below, $V_0$ and $\nu^+$ are the concordance invariants introduced respectively in \cite{Rasmussen:2003-1} and \cite{Hom-Wu:2016-1}. 

\begin{theorem}\label{thm:B} 
Let $p$ be a prime and $q>0$. If $S^3_{p/q}(K)$ has finite order in $\Theta_\mathbb{Q}^3 / \mathcal{L}$, then $$ 3V_0(K) +1 \leq p.$$ Furthermore, if $q \equiv -1 \pmod{p}$, then $V_0(K)=0$, or equivalently $\nu^+(K)=0$.
\end{theorem}

By restricting to prime integer surgeries on knots we get a better lower bound and an upper bound on the surgery coefficients. Note that Rasmussen~\cite{Rasmussen:2004-1} showed that if a non-trivial knot admits a lens space surgery of slope $p$, then $p \leq 4g_3(K)+3$, where $g_3$ is the Seifert genus. The following result should be thought of as the concordance analogue of~\cite[Theorem~1]{Rasmussen:2004-1}. We denote by $g_4(K)$ the slice genus of a knot $K$. 

\begin{theorem}\label{thm:C}
Let $p$ be a prime and $K$ a knot with $\nu^+(K) \neq 0$. If $S^3_p(K) \in \lens$, then $$4V_0(K) +1 \le p \le 4g_4(K)+3.$$
\end{theorem}
Recall that for $L$-space knots it follows from~\cite[Corollary 1.6]{Ozsvath-Szabo:2005-1} that $g_3 = g_4$, so the right-most inequality coincides with Rasmussen's bound in this case. The right inequality is sharp for all $(2,2k+1)$-torus knots, and the left inequality is sharp for the $(2,3)$-torus knot. 

Using Theorem~\ref{thm:B} and Theorem~\ref{thm:C}, we exhibit infinitely many irreducible $L$-spaces not $\Q$-homology cobordant to any connected sum of lens spaces (see Corollary~\ref{cor:L'/L} and Corollary~\ref{cor:gstarobstruction}). 
~\\

The results described in Theorem~\ref{thm:main} suggest the following question.
\begin{question}\label{question} Does every class in $\qhom$ contain an element $\widetilde{Y}$ such that if $Y$ is $\Q$-homology cobordant to $\widetilde{Y}$, then there is an injection $$H_1(\widetilde{Y}; \mathbb{Z}) \hookrightarrow H_1(Y;\mathbb{Z})?$$
\end{question}

Theorem~\ref{thm:main} provides a positive answer for the classes in $\lens$. It seems likely that this is a special property of the subgroup $\mathcal{L}$, and we think it would be interesting to find a class in $\qhom$ that gives a negative answer to Question~\ref{question}. One way to do this would be to exhibit two $\Q$-homology cobordant $\Q$-homology spheres $Y_1$ and $Y_2$, such that the orders of their first integral homology groups are relatively prime, and there exists no $\Z$-homology sphere in the same $\Q$-cobordism class.

\subsection*{Organization of the paper}
In Section \ref{sec:lens}, we recall some results from~\cite{Lisca:2007-2}, and state some consequences; we then use them to give a more precise statement for the first part of Theorem \ref{thm:main}. Section~\ref{sec:lattices} contains some preliminaries and technical results on integral lattices. The proof of Theorem~\ref{thm:main} is concluded in Section~\ref{sec:main'}. In Section~\ref{sec:consequences}, we state a non-splitability result for lens spaces  (Proposition~\ref{prop:nonsplit}) and prove the corollaries stated in the introduction. Finally, in Section~\ref{sec:applications}, we prove Theorems~\ref{thm:B} and \ref{thm:C}, and describe some of their consequences.

\subsection*{Notation and conventions} 
In this paper, every $3$-manifold is smooth, connected, closed, and oriented. All $4$-manifolds are smooth, connected, compact, and oriented. We indicate with $-M$ the manifold $M$ with reversed orientation and $-K$ the knot obtained by taking the mirror image of $K$ with the reversed orientation. The connected sum of $n$ copies of a manifold $M$ is denoted by $nM$ and the connected sum of $n$ copies of a knot $K$ is denoted by $nK$.

\subsection*{Acknowledgments}The authors would like to thank Charles Livingston and Min Hoon Kim for some helpful correspondence, Peter Feller and Paolo Lisca for interesting conversations, and Brendan Owens for some corrections on an earlier version of the paper. JP would also like to thank Jennifer Hom and Peter Lambert-Cole for helpful conversations.
 PA and DC acknowledge support from the European Research Council (ERC) under the European Unions Horizon 2020 research and innovation programme (grant agreement No 674978). PA and JP were partially supported by MPIM. Lastly, we are very grateful to the anonymous referees for their detailed and thoughtful suggestions.

\section{The lens space subgroup}~\label{sec:lens}
Recall that if $\gcd(p,q)=1$, the lens space $L(p,q)$ is the result of $-p/q$ Dehn surgery on the unknot. Up to orientation preserving diffeomorphism, we may assume that $p > q >0$. Moreover, there is an orientation preserving diffeomorphism between $-L(p,q)$ and $L(p,p-q)$. We now recall Lisca's classification of lens spaces up to $\Q$-homology cobordism. In \cite{Lisca:2007-1} the author defines a certain subset $\mathcal{R} \subset \Qp$. In \cite{Lisca:2007-2} the following family of subsets is introduced.
$$F_n := \left\{ \frac{m^2n}{mnk + 1} \mid m > k > 0, \gcd(m, k) = 1\right\} \subset \Q, \text{ $ $ }n \geq 2.$$
With this notation in place we can state the main result from \cite{Lisca:2007-2}.

\begin{theorem}[\cite{Lisca:2007-2}]\label{thm:Lisca2} A connected sum of lens spaces bounds a $\Q$-homology ball if and only if each summand is $($possibly orientation-reversing$)$ diffeomorphic to one of the following
\begin{enumerate}[font=\upshape]
\item $ L(p, q)$, $p/q \in \mathcal{R};$
\item $L(p, q) \# L(p, p - q);$
\item $L(p_1, q_1)\# L(p_2, q_2)$, $p_i/q_i \in F_2, i = 1, 2;$
\item $L(p, q) \# L(n, n - 1)$, $p/q \in F_n$ for some $n \geq 2;$
\item $L(p_1, q_1)\# L(p_2, p_2 - q_2)$, $p_i/q_i \in F_n, i = 1, 2,$ for some $n \geq 2$.\qed\end{enumerate}
\end{theorem}

Theorem~\ref{thm:main} guarantees the existence of a special representative contained in any given class of $\lens$. In order to characterize such elements, we introduce the following definition. 
\begin{definition}\label{def:reduced} A connected sum of lens spaces is said to be \emph{reduced} if the following conditions are satisfied.
\begin{enumerate}
\item there is no summand $\lp$ with $p/q \in \mathcal{R};$
\item there is no summand of the form $L(p, q) \# L(p, p - q);$
\item there is no summand of the form $L(p, q)$ with $p/q \in F_n$ or $p/(p-q) \in F_n$.\end{enumerate}
\end{definition}

\begin{proposition}\label{prop:ln1} The lens space $L(m,1)$ is reduced if and only if $m\neq 4$.
\end{proposition}
\begin{proof} By Definition~\ref{def:reduced}, we only need to check whether $m$ is contained in $\mathcal{R}$ or $F_n$ for some $n\ge2$. Note that for each $m\geq2$, $m$ and $m/(m-1)$ are not contained in any $F_n$ by definition.  Further, it is straightforward to check from \cite[Definition~$1.1$]{Lisca:2007-1} that $m$ is contained in $\mathcal{R}$ if and only if $m=4$, which completes the proof.
\end{proof}

Next, we describe a procedure to obtain the reduced representative in any class of $\mathcal{L}$.

\begin{proposition}\label{prop:reduced} For any $Y \in \mathcal{L}$, there exists a unique $($up to orientation preserving diffeomorphism$)$, reduced, possibly empty, connected sum of lens spaces $L_Y$ which is $\Q$-homology cobordant to $Y$. Moreover, any non-reduced connected sum of lens spaces $L$ that is $\Q$-homology cobordant to $L_Y$ satisfies $|H_1(L;\Z)|> |H_1(L_Y;\Z)|$.
\end{proposition}
\begin{proof}Let $L$ be a connected sum of lens spaces $\Q$-homology cobordant to $Y$. Whenever there is a summand $\lp$ with $p/q \in F_n$ (resp.\ $p/(p-q) \in F_n$), we can replace it with $L(n,1)$ (resp.\ $L(n,n-1)$) by using the relation $(4)$ from Theorem~\ref{thm:Lisca2}. Further, by using relations $(1)$ and $(2)$ from Theorem~\ref{thm:Lisca2}, it is clear that $L$ is $\Q$-homology cobordant to some reduced connected sum of lens spaces. 

Suppose now that $L_1$ and $L_2$ are reduced connected sum of lens spaces and that $L_1$ is $\Q$-homology cobordant to $L_2$. Then $L_1 \# -L_2$ bounds $\Q$-homology ball, and by Theorem~\ref{thm:Lisca2}, together with the fact that $L_1$ and $L_2$ are reduced, it is easy to see that $L_1 \# -L_2$ can be decomposed as a connected sum where each summand is of the form $L(p,q) \# L(p,p-q)$. Then it is clear that there is a orientation preserving diffeomorphism between $L_1$ and $L_2$. 

Lastly, suppose $L$ is a non-reduced connected sum of lens spaces $\Q$-homology cobordant to $L_Y$. Then, by the uniqueness of the reduced form we can apply the reduction process as described above to obtain $L_Y$ from $L$. The proof is completed by noting that each step strictly decreases the order of the first integral homology.\end{proof}

We call a lens space \emph{amphichiral} if it is orientation preserving diffeomorphic to its inverse, and \emph{chiral} otherwise. The following proposition follows easily  from Theorem~\ref{thm:Lisca2} and Definition~\ref{def:reduced}.

\begin{proposition}\label{prop:lensspacegp}There is an isomorphism $$\mathcal{L} \cong \mathbb{Z}^\infty \oplus \mathbb{Z}_2^\infty.$$ A basis is given by the set of reduced lens spaces. Moreover, the $\mathbb{Z}_2^\infty$ summand is generated by reduced amphichiral lens spaces.\qed
\end{proposition}

Using the above basis we can compute the following quotients of $\mathcal{L}$. 
\begin{proposition}\label{prop:lensspacegpmodulo} Let $p$ be a prime, and $\mathcal{L}_p =\mathcal{L} / \langle\{ L(r,s) \mid \gcd(r,p)=1 \}\rangle$. Then there is an isomorphism
$$ \mathcal{L}_p \cong 
\begin{cases}
\mathbb{Z}^\infty&\quad\text{if } p \equiv 3 \pmod{4}  \\

\mathbb{Z}^\infty \oplus \mathbb{Z}_2^\infty &\quad\text{if }p \equiv 1 \pmod{4} \text{ or } p=2.\\
\end{cases}$$
\end{proposition}
\begin{proof} By Proposition~\ref{prop:reduced} the subgroup generated by the lens spaces $L(r,s)$ such that $\gcd(r,p)=1$ is isomorphic to the subgroup generated by reduced lens spaces $L(r,s)$ where $\gcd(r,p)=1$. Then by Proposition~\ref{prop:lensspacegp}, the diffeomorphism classes of reduced lens spaces $L(p,q)$ with $\gcd(r,p)\neq 1$ provide a basis for $\mathcal{L}_p$. This quotient has a $\Z^\infty$ summand, since there are infinitely many reduced chiral lens spaces $L(r,s)$ where $\gcd(r,p)\neq 1$ (for instance, we can choose the family $\{L(p^i,1)\}_{i>2}$).

Recall that the lens space $L(r,s)$ is amphichiral if and only if $s^2 \equiv -1 \pmod{r}$. When $\gcd(r,p)\neq 1$, the equation $s^2 \equiv -1 \pmod{r}$ implies that $s^2 \equiv -1 \pmod{p}$. Moreover $-1$ is a quadratic residue modulo $p$ if and only if $p \equiv 1 \pmod{4}$ or $p=2$. Hence when $p \equiv 3 \pmod{4}$, there is no reduced amphichiral lens space $L(r,s)$ where $\gcd(r,p)\neq 1$, and we obtain the isomorphism $$\mathcal{L}_p \cong \Z^\infty.$$  

On the other hand, when $p \equiv 1 \pmod{4}$ or $p=2$, there are infinitely many such lens spaces, and we present a family for each case as follows. We first find an infinite family when $p \equiv 1 \pmod{4}$. Let $k$ be an even integer such that $p > k>0$ and $k^2 \equiv -1 \pmod{p}$. Let $r_m = \left(2mp+k\right)^2+1$ and $s_m = 2mp+k$ and consider the family of amphichiral lens spaces $\{L(r_m,s_m)\}_{m\geq 1}$. Since $k^2 \equiv -1 \pmod{p}$, we have $\gcd(r_m,p)=p$. We need to show that each lens space in this family is reduced by checking that conditions $(1), (2),$ and $(3)$ from Definition~\ref{def:reduced} are satisfied. 

Recall that the lens space $L(r,s)$ bounds a $\Q$-homology ball if and only if $r/s \in \mathcal{R}$ \cite{Lisca:2007-1} (see also Theorem~\ref{thm:Lisca2}). Furthermore, if a $\Q$-homology sphere bounds a $\Q$-homology ball, then the order of the first integral homology is a square \cite{Casson-Gordon:1978-1}. 
Condition $(1)$ is satisfied since $r_m$ is not a square of an integer and hence $r_m/s_m \not \in \mathcal{R}$; $(2)$ is automatically satisfied since we are only considering a single lens space. For $(3)$, we need to make sure that $r_m/s_m$ and $r_m/(r_m-s_m)$ are not contained in $F_n$ for any $n$. If $r/s \in F_n$ (resp.\ $r/(r-s) \in F_n$), then $L(r,s)$ is $\Q$-homology cobordant to $L(n,1)$ (resp.\ $L(n,n-1) \in F_n$) by Theorem~\ref{thm:Lisca2}. Then we see that $r_m/s_m$ and $r_m/(r_m-s_m)$ are not contained in $F_n$ for $n\neq 2,4$, since both $L(n,1)$ and $L(n,n-1)$ have infinite order in $\qhom$ by Proposition~\ref{prop:ln1} and Proposition~\ref{prop:lensspacegp}. For $n=2$, note that $r_m$ is an odd integer but the order of the first integral homology of any $\Q$-homology sphere that is $\Q$-homology cobordant to $L(2,1)$ is divisible by $2$. Lastly, for $n=4$, recall that $L(4,1)$ and $L(4,3)$ bound a $\Q$-homology ball, and we have already seen that $r_m/s_m \not \in \mathcal{R}$.

When $p=2$, let $r_m=\left(8m+3\right)^2+1$ and $s_m=8m+3$, and consider the family of amphichiral lens spaces $\{L(r_m,s_m)\}_{m\geq 1}$. It is easy to check that $\gcd(r_m,2)=2$, and using the same argument as before, we see that conditions $(1)$ and $(2)$ from Definition~\ref{def:reduced} are satisfied for each $L(r_m,s_m)$. Finally, for $(3)$ note that $L(r_m,s_m)$ is not $\Q$-homology cobordant to $L(n,1)$ for any $n$ since $2L(r_m,s_m)$ bounds a $\Q$-homology ball, and neither $r_m$ nor $2r_m$ is a square. This completes the proof.\end{proof}

Using Theorem~\ref{thm:Lisca2} we obtain an isomorphism  between the subgroup of $\mathcal{C}$ generated by $2$-bridge knots and the subgroup of $\lens$ generated by odd lens spaces. We denote by $K(p,q)$ the unique $2$-bridge knot such that $\Sigma_2 (K(p,q)) = L(p,q)$.

\begin{proposition}\label{prop:2bridgelens}
Let $\mathcal{B}$ be the subgroup of $\mathcal{C}$ generated by $2$-bridge knots. Then there exists an isomorphism
$$\beta_2{\big|}_{\mathcal{B}} \colon \mathcal{B} \rightarrow \langle\{ L(r,s) \mid r \text{ is odd} \}\rangle$$
where $\beta_{2}\colon \mathcal{C} \rightarrow \Theta_\mathbb{Q}^3$ is the homomorphism defined by taking the $2$-fold cyclic branched cover. \end{proposition}
\begin{proof}
We only need to show that $\beta_2{\big|}_{\mathcal{B}}$ is injective. Let $K = \#_i K(p_i,q_i)$ be an element of $\ker (\beta_2{\big|}_{\mathcal{B}})$; then $L = \#_iL(p_i,q_i)$ bounds a $\Q$-homology ball. By Theorem~\ref{thm:Lisca2}, we have $5$ cases to examine.\\ Case $(1)$: For some $i$, $\frac{p_i}{q_i} \in \mathcal{R}$. It follows from~\cite[Theorem~1.2]{Lisca:2007-1} that $K(p_i,q_i)$ is smoothly slice. Hence we can remove the summand and iterate the argument.\\
Case $(2)$: For some $i,j$, $L(p_j,q_j) = L(p_i,p_i-q_i)$. This implies that $K(p_j,q_j) = K(p_i,p_i-q_i) = - K(p_i,q_i)$, hence we can remove this pair.\\
Case $(3)$: This case does not occur since if $\frac{p}{q} \in F_2$ then $p$ is even.\\
Cases $(4),(5)$: These are covered by \cite[Lemma~3.4]{Lisca:2007-2}
\end{proof}

The next theorem is a more precise statement for the first part of Theorem~\ref{thm:main}. We prove this in Section~\ref{sec:main'}. We use the terminology from  Definition~\ref{def:reduced}.

\begin{theorem}\label{thm:main'}
For any $Y$ in $\mathcal{L}$, let $L_Y$ be the reduced connected sum of lens spaces $\Q$-homology cobordant to $Y$. Then there is an injection
 $$H_1(L_Y; \mathbb{Z}) \hookrightarrow H_1(Y;\mathbb{Z}).$$
\end{theorem}

We give a proof of Theorem~\ref{thm:main} assuming Theorem~\ref{thm:main'}.

\begin{proof}[Proof of Theorem~\ref{thm:main}]
For $Y$ as in Theorem~\ref{thm:main}, take $L=L_Y$ in Theorem~\ref{thm:main'}.
\end{proof}
%

\section{Lattices}\label{sec:lattices}
An \emph{integral lattice} is a pair $(G,Q)$, where $G$ is a finitely generated free abelian group and $Q$ is a $\mathbb{Z}$-valued symmetric bilinear form defined on $G$. We indicate with $(\mathbb{Z}
^N, -\Id)$ the integral lattice with the standard negative definite form. A \emph{morphism of integral lattices} is a homomorphism of abelian groups which preserves the form. We say that two integral lattices are \emph{isomorphic} if there exists a bijective morphism between them. 
To any given $4$-manifold $X$ we can associate the integral lattice $(H_2(X;\mathbb{Z})/\tors,Q_X)$, where  $Q_X$ is the intersection form on $X$.

Any lens space $L(p,q)$ arises as the boundary of a canonical negative definite plumbed $4$-manifold $P(p,q)$, which  can be described by the plumbing graph
$$
\begin{tikzpicture}[xscale=1.5,yscale=-0.5]
\node (A0_1) at (1, 0) {$-a_1$};
\node (A0_2) at (2, 0) {$-a_2$};
\node (A0_4) at (4, 0) {$-a_m$};
\node (A1_0) at (0, 1) {$\Gamma_{p,q}:=$};
\node (A1_1) at (1, 1) {$\bullet$};
\node (A1_2) at (2, 1) {$\bullet$};
\node (A1_3) at (3, 1) {$\dots$};
\node (A1_4) at (4, 1) {$\bullet$};
\path (A1_2) edge [-] node [auto] {$\scriptstyle{}$} (A1_3);
\path (A1_3) edge [-] node [auto] {$\scriptstyle{}$} (A1_4);
\path (A1_1) edge [-] node [auto] {$\scriptstyle{}$} (A1_2);
\end{tikzpicture}
$$

\noindent where the $a_i$'s are uniquely determined by the conditions $a_i\geq 2$ and
$$a_1 - \cfrac{1}{a_2 - \cfrac{1}{\ldots - \cfrac{1}{a_m}}}=\frac{p}{q}.$$

We denote the integral lattice associated with $P(p,q)$ as $(\Z\Gamma_{p,q},Q_{p,q})$, and call it the \emph{integral lattice associated with} $L(p,q)$.  The notation $\Z\Gamma_{p,q}$ is non standard. Note that we are actually taking the free abelian group generated by the vertex set of $\Gamma_{p,q}$. Since $-L(p,q)\cong L(p,p-q)$, we also obtain a \emph{dual} negative definite integral lattice $(\Z\Gamma_{p,p-q},Q_{p,p-q})$ associated with $L(p,p-q)$. We can extend the above terminology to connected sums of lens spaces. Given $\#_{i=1}^nL(p_i,q_i)$, the corresponding plumbed $4$-manifold is the boundary connected sum $\natural_{i=1}^nP(p_i,q_i)$. The associated integral lattice is
$$
\left(\Z\Gamma_{\#_{i=1}^nL(p_i,q_i)},Q_{\#_{i=1}^nL(p_i,q_i)}\right):=\bigoplus_{i=1}^n \left(\Z\Gamma_{p_i,q_i},Q_{p_i,q_i}\right).
$$

We now recall some definitions and results from~\cite{Lisca:2007-1,Lisca:2007-2}. Suppose we are given a morphism
\begin{equation}\label{eqn:morphism}
\left(\Z\Gamma_{\#_{i=1}^nL(p_i,q_i)},Q_{\#_{i=1}^nL(p_i,q_i)}\right)\hookrightarrow \left(\mathbb{Z}^N,-\Id\right),
\end{equation}
with $N=\rk(\Z\Gamma_{\#_{i=1}^nL(p_i,q_i)})$. The group $\mathbb{Z}\Gamma_{\#_{i=1}^nL(p_i,q_i)}$ is generated by the vertices of its corresponding plumbing graph $v_1,\ldots, v_N\in\mathbb{Z}^N$. Using the same notation for their image in $\mathbb{Z}^N$, we obtain a subset $S:=\{v_1,\ldots, v_N\} \subset \mathbb{Z}^N$. Most of the technical results from \cite{Lisca:2007-1,Lisca:2007-2} deal with specific properties of such subsets. Moreover, the main technical ingredient for the proof of Theorem~\ref{thm:main'} will also be stated in terms of these special subsets. This should justify the following abstract definitions from \cite{Lisca:2007-1,Lisca:2007-2}.

A subset $S=\{v_1,\ldots, v_N\} \subset \mathbb{Z}^N$ is said to be \emph{linear}, if for $i,j\in\{1,\ldots,N\}$,
$$
v_i\cdot v_j = 
\begin{cases}
-a_i\leq -2\quad& \text{if}\ i=j,\\
0\ \text{or}\ 1 \quad& \text{if}\ |i-j|=1,\\
0\quad& \text{if}\ |i-j|>1,
\end{cases}
$$
for some integers $a_i$, $i=1,\ldots,N$. Clearly, a subset $S \subset \mathbb{Z}^N$ is linear if and only if it is obtained from a morphism as in \eqref{eqn:morphism}. In particular, elements of $S$ correspond to vertices of a plumbing graph. The quantity $c(S)$ denotes the number of connected components of the associated linear graph. Also, recall the following definition:
$$I(S) := -\sum_{i=1}^N(v_i\cdot v_i +3).$$

An element $v_i \in S$ is said to be \emph{final} if the corresponding vertex has valence one. The standard basis of $(\mathbb{Z}^N,-\Id)$ is denoted by $\{e_1,\ldots, e_N\}$. We say that $e_i$ \emph{hits} $v \in \mathbb{Z}^N$ (or that $v$ hits $e_i$) if $e_i\cdot v \neq 0$. 
\begin{definition}\label{def:irreducible}
A subset $S\subset \mathbb{Z}^N$ is \emph{irreducible} if for any two elements $v, w\in S$ there
exists a finite sequence of elements of $S$
$$
v=v_0, v_1,\ldots, v_k=w, 
$$ such that for every $i=0,\ldots, k-1$ there exists some $e_{j_i}$ that hits both $v_i$ and $v_{i+1}$. A non-irreducible subset is called \emph{reducible}.
\end{definition}

Clearly, every subset $S \subset \mathbb{Z}^N$ can be written uniquely as the disjoint union of its irreducible components.

Given $e_i, v \in \mathbb{Z}^N$, we define $\pi_i(v) : = v - (v\cdot e_i)e_i$ to be \emph{the projection
of $v$ in the sub-lattice orthogonal to $e_i$}.

\begin{definition}\label{def:contraction} Let $S=\{v_1,\ldots, v_N\} \subset \mathbb{Z}^N$ be a linear subset such that $|v_i \cdot e_j| \leq 1$ for each $i,j\in\{1,\ldots,N\}$. Suppose there exist $1\leq h,s,t \leq N$ such that $e_h$ hits only $v_s$ and $v_t$ and $v_t \cdot v_t <-2$. Then, we say that the subset $S'\subseteq\langle e_1,\ldots,e_{h-1},
e_{h+1},\ldots, e_N\rangle\cong \mathbb{Z}^{N-1}$ defined by
\[
S':=\left( S\setminus\{v_s,v_t\}\right) \cup\{\pi_{h}(v_t)\}
\]
is obtained from $S$ by a \emph{contraction}, and we write $S\searrow
S'$. We also say that $S$ is obtained from $S'$ by an \emph{expansion}, and we write $S'\nearrow S$. If $v_s$ and $v_t$ are both final and $v_s\cdot v_s =-2$, we say that $S'$ is obtained from $S$ by a \emph{$-2$-final contraction} and $S$ is obtained from $S'$ by a \emph{$-2$-final expansion}.
\end{definition}

Our first step is to show that subsets originating from reduced connected sums of lens spaces can be described using only $-2$-final expansions.

\begin{proposition}\label{prop:contraction} Let $\#_{i=1}^nL(p_i,q_i)$ be a reduced connected sum of lens spaces. Then there exists a morphism
$$
(\Z\Gamma_{\#_{i=1}^nL(p_i,q_i)},Q_{\#_{i=1}^nL(p_i,q_i)})\oplus (\Z\Gamma_{\#_{i=1}^nL(p_i,p_i-q_i)},Q_{\#_{i=1}^nL(p_i,p_i-q_i)})\hookrightarrow (\mathbb{Z}^N,-\Id),$$ where $N=\rk(\Z\Gamma_{\#_{i=1}^nL(p_i,q_i)} \oplus \Z\Gamma_{\#_{i=1}^nL(p_i,p_i-q_i)})$.

Moreover, the subset $S$ associated with any such morphism can be decomposed as $S = S_1 \cup \ldots \cup S_n$ where each $S_i$ is irreducible and corresponds to a morphism 
$$
(\Z\Gamma_{p_i,q_i},Q_{p_i,q_i}) \oplus (\Z\Gamma_{p_i,p_i- q_i},Q_{p_i,p_i-q_i}) \hookrightarrow (\mathbb{Z}^{m_i},-\Id),
$$ where $m_i=\rk(\Z\Gamma_{p_i,q_i} \oplus \Z\Gamma_{p_i,p_i-q_i})$.
Finally, each $S_i$ is obtained from the subset $\{e_1+e_2,e_1-e_2\} \subset \mathbb{Z}^2$ via a sequence of $-2$-final expansions.

\end{proposition}
\begin{proof} The first assertion follows from a standard argument, which we sketch below. Since $$L:=\left(\#_{i=1}^nL(p_i,q_i)\right) \# \left(\#_{i=1}^nL(p_i,p_i-q_i)\right)\cong \left(\#_{i=1}^nL(p_i,q_i)\right) \# \left(-\#_{i=1}^nL(p_i,q_i)\right),$$
there exists a $\mathbb{Q}$-homology ball $W$ with $\partial W = L$. Let $$X := \left( \natural_{i=1}^nP(p_i,q_i)\right) \natural \left( \natural_{i=1}^nP(p_i,p_i-q_i)\right).$$ By Donaldson's diagonalization theorem~\cite{Donaldson:1987-1} the smooth closed $4$-manifold $X' = X \cup_L -W$ has standard negative definite intersection form. The inclusion $X \hookrightarrow X'$ induces the desired morphism of integral lattices (see also \cite{Lisca:2007-1}). 

From the above morphism we obtain a linear subset $S \subset \mathbb{Z}^N$. We claim that $S$ has no \emph{bad components}, i.e.\ $b(S)=0$ (see \cite[Definition 4.1]{Lisca:2007-1}\cite[Definition 4.4]{Lisca:2007-2} for the definition of a bad component). In fact, it follows from \cite[Lemma~$3.2$]{Lisca:2007-2} that a bad component corresponds to a lens space summand $L(p,q)$ with $p/q \in F_n$ for some $n\geq2$. Since we are assuming that $\#_{i=1}^nL(p_i,q_i)$ is reduced such summands do not occur. 

We can decompose $S$ as the disjoint union of its maximal irreducible components $S= \cup_{i=1}^k T_i$. It follows from \cite[Lemma~$2.7$]{Lisca:2007-1} that $I(S) =-2n$. Since $I(S)= \Sigma_{i=1}^k I(T_i)$ and $k \leq 2n$, there exists $T_i$ for some $i$ such that $I(T_i)<0$. Moreover, since $b(S)=0$, we see that $b(T_i) = 0$. It then follows from \cite[Proposition~$4.10$]{Lisca:2007-2} that $c(T_i)\leq 2$. It is implicit in \cite{Lisca:2007-1} (see also \cite[Lemma~$4.3$]{Aceto-Golla:2017-1}) that if $c(T_i)=1$ and $I(T_i)<0$ then the corresponding lens space bounds a $\mathbb{Q}$-homology ball. Since we are assuming that $\#_{i=1}^nL(p_i,q_i)$ is reduced, this is not possible, and we conclude that $c(T_i)=2$.

Now, the argument given in the proof of the main theorem in \cite{Lisca:2007-2} (more specifically the first subcase of the first case) applies to $T_i$. In particular, \cite[Lemma~$4.7$]{Lisca:2007-2} can be applied to $T_i$, and we see that $T_i$ is obtained from the subset $\{e_1+e_2,e_1-e_2\} \subset \mathbb{Z}^2$ via a sequence of $-2$-final expansions and the corresponding connected sum of lens spaces is of the form $L(p_i,q_i) \# L(p_i,p_i-q_i)$.

Finally, note that 
$$I(S \setminus T_i) = -2(n-1),\; b(S \setminus T_i)=0,\;c(S \setminus T_i)=2(n-1).$$ By induction on $|S|$, the conclusion follows.
\end{proof}

Let $\Gamma$ be a sublattice of an integral lattice $\Gamma'$. Then the orthogonal complement of $\Gamma$ in $\Gamma'$ is defined as follows,
$$ \Gamma^\perp:= \{v \in \Gamma' \mid v\cdot w =0 \text{ for all } w \in \Gamma\}.$$ For any $S \subset \mathbb{Z}^N$, the sublattice generated by $S$ is denoted by $\langle S\rangle$.

\begin{lemma}\label{lemmal:key} Let $S \subset \mathbb{Z}^N$ be a linear subset obtained from the subset $\{e_1+e_2,e_1-e_2\} \subset \mathbb{Z}^2$ via a sequence of $-2$-final expansions. Write $S=S_1 \cup S_2$ where each $S_i$ corresponds to a connected component of the graph associated to $S$. Then 
\begin{equation}\label{eqn:key}
\langle S_1 \rangle^\perp = \langle S_2 \rangle \text{ and } \langle S_2 \rangle^\perp = \langle S_1 \rangle.
\end{equation}
\end{lemma}
\begin{proof} We proceed by induction on the number of $-2$-final expansion. If $S =\{e_1+e_2,e_1-e_2\}$, then we can set $S_1 = \{e_1+e_2\}$ and $S_2 = \{e_1-e_2\}$. Now suppose \eqref{eqn:key} holds for some subset $S \subset \mathbb{Z}^N$ with $S = S_1 \cup S_2$. Let $S' \subset \mathbb{Z}^{N+1}$ be a subset obtained from $S$ via a single $-2$-final expansion. Write $S' = S_1' \cup S_2'$, where $S_i'$ corresponds to a connected component of the graph associated to $S'$. We may assume without loss of generality that $|S_1'| = |S_1|$ and $|S_2'| = |S_2|+1$. By abuse of notation we can think of $S = S_1\cup S_2$ as a subset of $\mathbb{Z}^{N+1}$. By definition of $-2$-final expansion, there exist two elements $w' \in S_1'$ and $v' \in S_2'$ such that $w'= w+e_{N+1}$ for some $w\in S_1$ and $v'=e_N+\varepsilon e_{N+1}$ where $\varepsilon = w\cdot e_N$. Note that $\varepsilon$ is either $1$ or $-1$. Moreover, $e_N$ only hits three vectors in $S'$ namely, $w', v',$ and one more vector in $S_2'$. 

First, we prove that $\langle S_1' \rangle^\perp = \langle S_2' \rangle$. Clearly, $\langle S_1' \rangle^\perp \supseteq \langle S_2' \rangle$ so it is enough to show that $\langle S_1' \rangle^\perp \subseteq \langle S_2' \rangle$. Suppose $x \in \langle S_1'\rangle^\perp$, then we can write $x=\widetilde{x} + a_{N+1}e_{N+1}$ where $\widetilde{x} \cdot e_{N+1}=0$. Let $y = \widetilde{x} - \varepsilon(\widetilde{x} \cdot w)e_N$, so that $$x=y+\varepsilon(\widetilde{x}\cdot w)e_N+a_{N+1}e_{N+1} = y + \varepsilon(\widetilde{x}\cdot w)v'.$$ Here we used the fact that $x\cdot w' = 0$ which implies that $a_{N+1}=\widetilde{x}\cdot w$. Since, $v'\in\langle S_2'\rangle$, it is enough to show that $y \in \langle S_2'\rangle$. In fact, we show that $y \in \langle S_2 \rangle$. Note that by the inductive hypothesis and the fact that $y$ is a linear combination of $\{ e_1, \ldots, e_N\}$, it is enough to show that $y\cdot z =0$ for all $z \in S_1$. Clearly, $y\cdot w =0$. Let $z\in S_1 \setminus \{w\}$ and note that $z$ is also an element of $S_1'$, in particular $e_N$ and $e_{N+1}$ do not hit $z$. Then we see that 
$$y\cdot z =\left(\widetilde{x}-\varepsilon(\widetilde{x}\cdot w)e_N\right)\cdot z=\widetilde{x}\cdot z = x\cdot z = 0.$$

Now, we prove that $\langle S_2' \rangle^\perp = \langle S_1' \rangle$. As before, it is clear that $\langle S_2' \rangle^\perp \supseteq \langle S_1' \rangle$. We show that $\langle S_2' \rangle^\perp \subseteq \langle S_1' \rangle$. Suppose $x \in \langle S_2'\rangle^\perp$, then we can write $x=\widetilde{x} + a_{N+1}e_{N+1}$ where $\widetilde{x} \cdot e_{N+1}=0$. Since $\langle S_2 \rangle \subset \langle S_2' \rangle$, and each vector in $S_2$ does not hit $e_{N+1}$, we see that $\widetilde{x}\cdot z = 0$ for all $z \in \langle S_2 \rangle$. By the inductive hypothesis $\widetilde{x} \in \langle S_1 \rangle$. Now, we can write $\widetilde{x} = \widehat{x} + cw$ where $\widehat{x} \in \langle S_1 \setminus \{w\} \rangle$. Since
$$x\cdot v' =\left(\widehat{x}+cw+a_{N+1}e_{N+1}\right)\cdot \left( e_N + \varepsilon e_{N+1} \right) = \varepsilon c - \varepsilon a_{N+1}=0 ,$$ we have $c=a_{N+1}$. It follows that $x$ can be written as a linear combination of two elements of $\langle S_1' \rangle$, namely,
$$x =\widetilde{x} + a_{N+1}e_{N+1} = \widehat{x} + c(w+e_{N+1}) = \widehat{x} + cw'.$$ This concludes the proof.\end{proof}

\section{Proof of Theorem~\ref{thm:main'}}\label{sec:main'}
In this section we prove Theorem~\ref{thm:main'}. Our main ingredient is the following proposition.

\begin{proposition}\label{prop:isomorphic}Let $L$ be a reduced connected sum of lens spaces, and $P$ be the canonical negative definite plumbed $4$-manifold associated to it. Let $Y$ be a $\Q$-homology sphere which is $\Q$-homology cobordant to $L$ via a cobordism $W$. Finally, let $X:=P \cup_L W$. Then the integral lattices $(H_2(P;\mathbb{Z}),Q_P)$ and $(H_2(X;\mathbb{Z}),Q_X)$ are isomorphic.
\end{proposition}
\begin{proof} Let us write $L= \#_{i=1}^nL(p_i,q_i)$ and $N=\rk(\Z\Gamma_{\#_{i=1}^nL(p_i,q_i)} \oplus \Z\Gamma_{\#_{i=1}^nL(p_i,p_i-q_i)})$. Let $P^*$ be the canonical negative definite plumbing associated to $-L$ and let $$\overline{X} := P \cup_L W \cup_Y -W \cup_{-L} \cup P^*.$$ By Donaldson's diagonalization theorem \cite{Donaldson:1987-1}, $H_2(\overline{X},Q_{\overline{X}}) \cong (\mathbb{Z}^N, -\Id)$. The inclusions $P \hookrightarrow X \hookrightarrow \overline{X}$ induce morphisms $$
(\Z\Gamma_{\#_{i=1}^nL(p_i,q_i)},Q_{\#_{i=1}^nL(p_i,q_i)})\hookrightarrow (H_2(X;\mathbb{Z}),Q_X) \hookrightarrow (\mathbb{Z}^N,-\Id).$$ Moreover the inclusion $P^* \hookrightarrow \overline{X}$ induces a morphism $$
(\Z\Gamma_{\#_{i=1}^nL(p_i,p_i-q_i)},Q_{\#_{i=1}^nL(p_i,p_i-q_i)})\hookrightarrow (\mathbb{Z}^N,-\Id).$$ 

By abuse of notation we will identify all these lattices with their image in the standard lattice. Clearly, $(H_2(X;\mathbb{Z}),Q_X)$, viewed as a sublattice of $\mathbb{Z}^N$ is contained in $\left(\Z\Gamma_{\#_{i=1}^nL(p_i,p_i-q_i)},Q_{\#_{i=1}^nL(p_i,p_i-q_i)}\right)^\perp$, therefore we have $$ \left( \Z\Gamma_{\#_{i=1}^nL(p_i,q_i)}, Q_{\#_{i=1}^nL(p_i,q_i)} \right) \subseteq \left( H_2(X;\mathbb{Z}),Q_X \right) \subseteq \left(\Z\Gamma_{\#_{i=1}^nL(p_i,p_i-q_i)},Q_{\#_{i=1}^nL(p_i,p_i-q_i)}\right)^\perp.$$ In order to conclude the proof it is enough to show that $$ \left( \Z\Gamma_{\#_{i=1}^nL(p_i,q_i)}, Q_{\#_{i=1}^nL(p_i,q_i)} \right) = \left(\Z\Gamma_{\#_{i=1}^nL(p_i,p_i-q_i)},Q_{\#_{i=1}^nL(p_i,p_i-q_i)}\right)^\perp.$$ The inclusion $P\cup P^* \subset \overline{X}$ induces the morphism $$
(\Z\Gamma_{\#_{i=1}^nL(p_i,q_i)},Q_{\#_{i=1}^nL(p_i,q_i)})\oplus (\Z\Gamma_{\#_{i=1}^nL(p_i,p_i-q_i)},Q_{\#_{i=1}^nL(p_i,p_i-q_i)})\hookrightarrow (\mathbb{Z}^N,-\Id).$$
Let $S$ be the corresponding linear subset. By Proposition~\ref{prop:contraction}, we can write $\mathbb{Z}^N=\mathbb{Z}^{m_1} \oplus \cdots \oplus \mathbb{Z}^{m_n}$ and decompose the above morphism as follows $$
(\Z\Gamma_{p_i,q_i},Q_{p_i,q_i}) \oplus (\Z\Gamma_{p_i,p_i- q_i},Q_{p_i,p_i-q_i}) \hookrightarrow (\mathbb{Z}^{m_i},-\Id),
$$ where $m_i=\rk(\Z\Gamma_{p_i,q_i} \oplus \Z\Gamma_{p_i,p_i-q_i})$. Note that $S=\cup_{i=1}^n S_i$ and by Proposition~\ref{prop:contraction} each $S_i$ is obtained from the subset $\{e_1+e_2,e_1-e_2\} \subset \mathbb{Z}^2$ via a sequence of $-2$-final expansions. \\For each $i$, let $S_i = S_i^1 \cup S_i^2$ where $\langle S_i^1 \rangle = \Z\Gamma_{p_i,q_i}$ and $\langle S_i^2 \rangle = \Z\Gamma_{p_i,p_i-q_i}$. By Lemma~\ref{lemmal:key}, we have $\Z\Gamma_{p_i,q_i}= \left(\Z\Gamma_{p_i,p_i-q_i}\right)^\perp$ in $\mathbb{Z}^{m_i}$ for each $i$. 

Note that if $x \in \left(\Z\Gamma_{\#_{i=1}^nL(p_i,p_i-q_i)},Q_{\#_{i=1}^nL(p_i,p_i-q_i)}\right)^\perp$, then we can write $x = x_1 + \cdots + x_n$, where $x_i \in \mathbb{Z}^{m_i}$ for each $i$. Moreover, $x_i \in \left(\Z\Gamma_{p_i,p_i-q_i}\right)^\perp$ in $\mathbb{Z}^{m_i}$. It is clear then that $$ \left( \Z\Gamma_{\#_{i=1}^nL(p_i,q_i)}, Q_{\#_{i=1}^nL(p_i,q_i)} \right) = \left(\Z\Gamma_{\#_{i=1}^nL(p_i,p_i-q_i)},Q_{\#_{i=1}^nL(p_i,p_i-q_i)}\right)^\perp,$$ 
which concludes the proof.
\end{proof}

\begin{proof}[Proof of Theorem~\ref{thm:main'}]
Let $P$ be the canonical negative definite plumbing associated to $L_Y$. Let $W$ be a $\mathbb{Q}$-homology cobordism from $Y$ to $L_Y$ and let $X:=P \cup_L W$. Then by Proposition~\ref{prop:isomorphic} the integral lattices $(H_2(P;\mathbb{Z}),Q_P)$ and $(H_2(X;\mathbb{Z}),Q_X)$ are isomorphic.
By abuse of notation, $Q_P$ (respectively $Q_X$) denotes a matrix representation of the intersection pairing on $H_2(P;\mathbb{Z})$ (respectively $H_2(X;\mathbb{Z})/\tors$).
Note that $Q_P$ gives a presentation matrix for $H_1(L_Y; \Z)$. In particular, we have
\begin{equation}\label{eqn:first}
H_1(L_Y ; \mathbb{Z}) \cong \Z^N / Q_P(\Z^N) \cong \Z^N / Q_X(\Z^N),
\end{equation} where $N=\rk(H_2(P;\mathbb{Z}))=\rk(H_2(X;\mathbb{Z})/\tors)$. Now we claim that there exists an injection
\begin{equation}\label{eqn:second}
\Z^N / Q_X(\Z^N) \hookrightarrow H_1(Y;\Z) / T,
\end{equation} for some subgroup $T \subset  H_1(Y;\Z)$. To prove the claim we follow Owens and Strle's argument from \cite[Section~$2$]{Owens-Strle:2006-1}.

Consider the following exact sequence of the pair $(X,Y)$ with integral coefficients,

\addtolength{\arraycolsep}{-3pt}
$$\begin{array}{cccccccccccc}
0\longrightarrow&H_2(X)&\stackrel{j_*}{\longrightarrow}&H_2(X,Y)&\stackrel{\phi}\longrightarrow&H_1(Y)
&\longrightarrow&H_1(X)&\longrightarrow&H_1(X,Y)&\longrightarrow0,\\ &\|&&\|&&
&&&&&\\ &\Z^N\oplus T_2&&\Z^N\oplus T_1&& && && &
\end{array}$$ 
where $T_1$ and $T_2$ are torsion subgroups. We can choose bases for the free part of $H_2(X;\mathbb{Z})$ and $H_2(X,Y;\mathbb{Z})$ so that 
$$j_*=\left(
\begin{array}{cc}
Q_X & 0\\
* & \tau
\end{array}
\right),$$
where $\tau\colon T_2 \hookrightarrow T_1$ is an injection. Then it can be easily checked that there exists an injection from $\Z^N / Q_X(\Z^N)$ to $H_1(Y;\Z) / T$ induced by $\phi$, where $T = \phi(T_1)$.

Finally, we leave it to the reader to verify that if $G$ and $G'$ are finite abelian groups, such that there exists a surjection from  $G$ to $G'$, then there exists an injection from $G'$ to $G$. In particular, there exists an injection
\begin{equation}\label{eqn:third}
H_1(Y;\Z) / T \hookrightarrow H_1(Y;\Z).
\end{equation}
Combining \eqref{eqn:first}, \eqref{eqn:second}, \eqref{eqn:third} concludes the proof. \end{proof}

\section{Consequences of Theorem~\ref{thm:main}}\label{sec:consequences}
This section contains the proofs of some of the corollaries of Theorem~\ref{thm:main} stated in the introduction, together with some other related results. We first prove that the natural maps mentioned in the introduction have infinitely generated cokernels.

\begin{proof}[Proof of Corollary \ref{cor:AK}]Let $Y$ be an integral homology sphere that is contained in $\mathcal{L}$. By Theorem~\ref{thm:main} there exists a connected sum of lens spaces $L$ such that $H_1(L;\mathbb{Z})$ injects into $H_1(Y;\mathbb{Z}) = 0$. It follows that $H_1(L;\mathbb{Z}) = 0$ and $L = S^3$. The second part of the corollary easily follows since $\mathcal{L}$ injects into $\Coker \psi$, and it follows from Proposition~\ref{prop:lensspacegp} that $\mathcal{L} \cong \mathbb{Z}^\infty \oplus \mathbb{Z}_2^\infty$.\end{proof}

\begin{proof}[Proof of Corollary~\ref{cor:Zp}] Let $Y \in \psi_p(\Theta_{\mathbb{Z}_p}^3) \cap \mathcal{L}$. Then by Theorem~\ref{thm:main} there exists a connected sum of lens spaces $L$ such that $H_1(L;\mathbb{Z})$ injects into $H_1(Y;\mathbb{Z})$. In particular, the order of $H_1(L;\mathbb{Z})$ divides the order of $H_1(Y;\mathbb{Z})$. This implies that $L$ is a connected sum of lens spaces such that each summand has no $p$-torsion in the first integral homology group. The second part of the corollary follows from Proposition~\ref{prop:lensspacegpmodulo}.\end{proof}

Now we prove the divisibility condition for knots concordant to a connected sum of $2$-bridge knots.

\begin{proof}[Proof of Corollary~\ref{cor:2bridge}] 
Let $J'$ be a knot that is concordant to a connected sum of $2$-bridge knots. It follows that the double branched cover $\Sigma_2(J')$ is $\Q$-homology cobordant to a connected sum of lens spaces. By Theorem~\ref{thm:main}, there exists a connected sum of lens spaces $L$, such that if $Y$ is $\Q$-homology cobordant to $L$, then there is an injection $H_1(L; \mathbb{Z}) \hookrightarrow H_1(Y;\mathbb{Z}).$ In particular, $|H_1(L; \mathbb{Z})|$ divides $|H_1(\Sigma_2(J');\mathbb{Z})| = \det(J')$. Since $\det(J')$ is odd, so is $|H_1(L;\mathbb{Z})|$. Therefore $L$ is the double branched cover of a connected sum of $2$-bridge knots, say $K$. It follows from Proposition~\ref{prop:2bridgelens} that a connected sum of $2$-bridge knots is slice if and only if its branched double cover bounds a $\Q$-homology ball.

This implies that $K$ is concordant to $J'$. Now, if $J$ is concordant to $K$, then the double branched cover $\Sigma_2(J)$ is $\Q$-homology cobordant to $L$. Further, $|H_1(L; \mathbb{Z})| = \det(K)$ divides $|H_1(\Sigma_2(J);\mathbb{Z})| = \det(J)$.

Finally, suppose there is a connected sum of $2$-bridge knots $K'$ that is concordant to $K$ satisfying the same property. Then it is clear from the property of $L$ that $H_1(L; \mathbb{Z}) \cong H_1(\Sigma_2(K');\mathbb{Z})$. Then by Theorem~\ref{thm:main} there is an orientation preserving diffeomorphism from $L = \Sigma_2(K)$ to $\Sigma_2(K')$. This implies that $K$ and $K'$ are isotopic \cite{Hodgson-Rubinstein:1985-1}.
\end{proof}

We characterize which connected sums of two torus knots are concordant to a connected sum of $2$-bridge knots.

\begin{proof}[Proof of Corollary~\ref{cor:torus}] By~\cite[Corollary 1.5]{Aceto-Alfieri:2017-1} we only need to show that $T_n:= T_{3,6n+1} \# - T_{3,6n+2}$ is not concordant to a connected sum of $2$-bridge knots for any $n\geq 1$. If it is, then by Corollary~\ref{cor:2bridge} there exists a connected sum of $2$-bridge knots $K$ concordant to $T_n$ and such that $\det(K)$ divides $\det(T_n)=3$. This implies that $K$ is either the unknot, $T_{2,3}$, or $-T_{2,3}$. This is not possible, since torus knots are linearly independent in the knot concordance group~\cite{Litherland:1979-1}.\end{proof}

We can now state and prove a more general version of Corollary~\ref{cor:nonsplitta} (recall that $L(ab,1)$ is reduced if and only if $ab\neq 4$). Note that the hypothesis of our result and those from \cite{Kim-Livingston:2014-1} are disjoint. In the next proposition, the case $(a,b) = (2,2)$ has to be excluded since $L(4,1)$ is $\Q$-homology cobordant to $L(2,1) \# L(2,1)$.

\begin{proposition}\label{prop:nonsplit}Let $L(ab,r)$ be a reduced lens space where $a$ and $b$ are not relatively prime. Then $L(ab,r)$ is not $\Q$-homology cobordant to $Y_1 \# Y_2$, where $H_1(Y_1;\Z) = \Z_a$ and $H_1(Y_2;\Z) = \Z_b$. In particular, $L(ab,1)$ is not $\Q$-homology cobordant to $Y_1 \# Y_2$, where $H_1(Y_1;\Z) = \Z_a$ and $H_1(Y_2;\Z) = \Z_b$ if $a$ and $b$ are not relatively prime, and $(a,b) \neq (2,2)$.
\end{proposition}

\begin{proof}Suppose that $L(ab,r)$ is $\Q$-homology cobordant to $Y_1 \# Y_2$, where $H_1(Y_1;\Z) = \Z_a$ and $H_1(Y_2;\Z) = \Z_b$. By Theorem~\ref{thm:main'} there is an injection
$$H_1(L(ab,r); \mathbb{Z}) \cong \mathbb{Z}_{ab} \hookrightarrow H_1(Y_1;\mathbb{Z}) \oplus H_1(Y_2;\mathbb{Z}) \cong \mathbb{Z}_a\oplus \mathbb{Z}_b.$$
This is not possible unless $a$ and $b$ are relatively prime. By Proposition~\ref{prop:ln1}, $L(n,1)$ is a reduced lens space for each $n\neq 4$, and the second part of the statement follows.
\end{proof}

We end this section by proving Corollary~\ref{cor:filtration}, which is a generalization of Proposition~\ref{prop:nonsplit}.

\begin{proof}[Proof of Corollary~\ref{cor:filtration}] Let $\{p_i\}_{i\geq 1}$ be an infinite family of distinct primes where $p_i\neq 5$ and $p_i \equiv  1 \pmod{4}$ for each $i$. Consider the family of chiral lens spaces $\{L(p_i^{n+1},1)\}_{i\geq 1}$ and amphichiral lens spaces $\{L(5p_i^{n+1},k_i)\}_{i\geq 1}$ where $k_i^2 \equiv -1 \pmod{5p_i^{n+1}}$. Recall that $-1$ is a quadratic residue modulo an odd prime $p$ if and only if $p \equiv 1 \pmod{4}$. Moreover, it can be easily checked that for a positive integer $m$ and an odd prime $p$, if $p \equiv 1 \pmod{4}$, then $-1$ is a quadratic residue modulo $p^m$. Lastly, for prime integers $p$ and $q$, if $a \equiv  b \pmod{p}$ and  $a \equiv  b \pmod{q}$, then $a \equiv  b \pmod{pq}$. Hence such $k_i$'s exist.

We claim that any linear combination of the form
\begin{equation}\label{eqn:linearcombi} \left(\#_{i=1}^{m_1} a_iL(p^{n+1}_i,1)\right)\#\left(\#_{i=1}^{m_2} L(5p^{n+1}_i,k_i)\right)\end{equation}
is reduced. It is straightforward to see that condition $(2)$ from Definition~\ref{def:reduced} is satisfied by construction. $(1)$ and $(3)$ are also satisfied for the first summand by Proposition~\ref{prop:ln1}. Since $5p_i^{n+1}$ is not a square, $(1)$ is satisfied for the second summand. Finally, observe that $L(5p^{n+1}_i,k_i)$ is not $\Q$-homology cobordant to $L(n,1)$ for any $n$, since $2L(5p^{n+1}_i,k_i)$ bounds a $\Q$-homology ball and neither $5p^{n+1}$ nor $10p^{n+1}$ is a square, and therefore condition $(3)$ holds.

Each element from $\{L(p_i^{n+1},1)\}_{i\geq 1}$ and $\{L(5p_i^{n+1},k_i)\}_{i\geq 1}$ is contained in $\mathcal{O}_{n+1}$. Then the proof is completed by observing that any linear combination of the form shown in \eqref{eqn:linearcombi} cannot be $\Q$-homology cobordant to any element in $\mathcal{O}_{n}$ by Theorem~\ref{thm:main'}.\end{proof}

\section{Obstructions from knot Floer homology}\label{sec:applications}

Theorem~\ref{thm:main} gives a strong restriction for $\Q$-homology spheres to be contained in $\mathcal{L}$. In this section, we give further restrictions, by combining Theorem~\ref{thm:main} with obstructions coming from Heegaard Floer homology.

In order to prove Theorem~\ref{thm:B}, we need a few technical results. For a $\Q$-homology sphere $Y$, we denote by $\lambda(Y)$ its Casson-Walker invariant~\cite{Walker:1992-1}.

\begin{proposition}[\cite{Walker:1992-1}]\label{proposition:cassonwalkerrecursive} Let $p>q>0$ be relatively prime integers, then \[\pushQED{\qed}\lambda(L(p,q)) = \frac{1}{4} -\frac{p^2+q^2+1}{12pq}-\lambda(L(q,p)).\qedhere\]
\end{proposition}

\begin{proposition}[{\cite[Proposition~2.4]{Rasmussen:2004-1}}]\label{proposition:cassonwalker} Let $p>q>0$ be relatively prime integers. If $$\lambda(L(p,q)) - \lambda(L(p,1)) \leq \frac{1}{4}\left(\frac{p-4}{4}\right),$$ then $q \in \{1,2,3\}$.\qed
\end{proposition} 

\begin{lemma}\label{lem:1}
If $p>q>0$ are relatively prime, then $$\lambda(L(p,1)) \le \lambda(L(p,q)) \leq \lambda(L(p,p-1)).$$
\end{lemma}
\begin{proof} It is straightforward to check that the inequality holds for $p=2$ and $p=3$. Suppose $p \geq 4$. If $q \not\in \{1,2,3\}$, it follows from Proposition~\ref{proposition:cassonwalker} that $\lambda(L(p,1)) \le \lambda(L(p,q))$. When $q \in \{1,2,3\}$, by direct computation using Proposition~\ref{proposition:cassonwalkerrecursive} we get $\lambda(L(p,1)) \le \lambda(L(p,q))$ for all $q$. Then the conclusion follows, since $\lambda(Y) = - \lambda(-Y)$.
\end{proof}

Given a $\Q$-homology sphere $Y$ and $\s\in \,$Spin$^c (Y)$, there is an associated spin$^c$ $\Q$-homology cobordism invariant $d(Y,\s) \in \Q$, called the correction term \cite{Ozsvath-Szabo:2003-1}. 
The Casson-Walker invariants of lens spaces can be computed using correction terms as follows. 
\begin{proposition}[{\cite[Lemma~2.2]{Rasmussen:2004-1}}]\label{proposition:cassonwalkerdinvt}
Let $p>q>0$ be relatively prime integers, then $$p\lambda(L(p,q)) = \sum_{\mathfrak{s}} d(L(p,q), \mathfrak{s}).$$ where the sum runs over all spin$^c$ structures on $L(p, q)$. \qed
\end{proposition}

Recall that given a knot $K \subset S^3$, there is an associated sequence of non-negative integers $\{V_i (K)\}_{i \ge 0}$, introduced by Rasmussen \cite{Rasmussen:2003-1}. Each $V_i$ is a smooth concordance invariant. Here we state some key properties of this sequence.

\begin{proposition}[{\cite[Proposition 1.6 and Remark 2.10]{Ni-Wu:2015-1}}]\label{proposition:Viformula}Let $p,q>0$ be relatively prime integers and $\ell \in \Z_p$. For any knot $K$,  \[ \pushQED{\qed}d(S^3_{p/q}(K),\ell) = d(S^3_{p/q}(U),\ell)-2\max \{V_{\lfloor\frac{\ell}{q}\rfloor}(K),V_{\lfloor\frac{p+q-1-\ell}{q}\rfloor}(K)\}. \qedhere \]
\end{proposition}

Here we are using the identification of $\operatorname{Spin}^c(S^3_{p/q}(K))\to\Z_p$ given in \cite{Ozsvath-Szabo:2011-1}, and $U$ denotes the unknot.

\begin{proposition}[{\cite[Proposition 7.6]{Rasmussen:2003-1}}]\label{proposition:Videcreasing}
For each $i \ge0$, \[\pushQED{\qed} V_i(K) -1 \leq V_{i+1}(K) \leq V_i(K).\qedhere\]
\end{proposition}

Another related concordance invariant $\nu^+\in \Z_{\ge0}$ was defined in \cite{Hom-Wu:2016-1}; it has the following property. 

\begin{proposition}[{\cite[Proposition 2.3]{Hom-Wu:2016-1}}]\label{proposition:whennu+iszero} For any knot $K$, $\nu^+(K)\geq 0$, and the equality holds if and only if $V_0(K)=0$. \qed
\end{proposition}

Now we are ready to prove Theorem~\ref{thm:B}.

\begin{proof}[Proof of Theorem~\ref{thm:B}]
Assume that $nS^3_{p/q}(K) \in \mathcal{L}$ for some positive integer $n$. Then by Theorem~1.1, $nS^3_{p/q}(K) \in \mathcal{L}$ is $\Q$-homology cobordant to a reduced connected sum $\#_{i=1}^{n'}L(p,q_i)$ for some integers $n'$ and $q_i$, where $n' \leq n$. 

If we call $Y=nS^3_{p/q}(K) \# (\#_{i=1}^{n'}L(p,p-q_i))$, then $Y$ bounds a $\Q$-homology ball. We can apply \cite[Proposition~4.1]{Owens-Strle:2006-1}, and deduce the existence of a metabolizer $M < H_1 (Y;\Z)$ and a spin$^c$ structure $\mathfrak{s}$ on $Y$ such that $d(Y, \mathfrak{s} + m) = 0$ for all $m \in M$. Since $p$ is a prime integer, the projection map from the metabolizer $M$ to each cyclic summand of $H_1 (Y;\Z)$ is surjective (see \cite[Corollary~3]{Kim-Livingston:2014-1}). We then consider the sum of all the correction terms for spin$^c$ structures extending to the $\Q$-homology ball. 

Using Proposition~\ref{proposition:Viformula} and the fact that the correction terms are additive under connected sums \cite[Theorem 4.3]{Ozsvath-Szabo:2003-1}, we obtain the following.
$$n \sum_{\ell = 0}^{p-1}d(S^3_{p/q}(U),\ell )  -2n\sum_{\ell = 0}^{p-1} \max \left\{ V_{\lfloor\frac{\ell}{q} \rfloor} (K), V_{\lfloor\frac{p+q-1-\ell}{q} \rfloor} (K)\right\} + \sum_{i = 0}^{n'}\sum_{\ell = 0}^{p-1} d(L(p,p-q_i), \ell) =0
$$

Recall that we are adopting the convention that $L(p,q)$ is $-p/q$ Dehn surgery on the unknot. Using Proposition~\ref{proposition:cassonwalkerdinvt} we get
\begin{equation}\label{eq:key}np \lambda(L(p,p-q))+ \sum_{i = 0}^{n'} p\lambda(L(p,p-q_i)) = 2n\sum_{\ell = 0}^{p-1} \max \left\{ V_{\lfloor\frac{\ell}{q} \rfloor} (K), V_{\lfloor\frac{p+q-1-\ell}{q} \rfloor} (K)\right\}
.\end{equation} When $p=2$, Equation~\eqref{eq:key} simplifies to $V_0(K)+V_1(K)=0$, and this implies that $V_0(K)=0$.

We can now assume that $p$ is odd. Using Proposition~\ref{proposition:Videcreasing}, it is straightforward to verify the following inequality
$$\sum_{\ell = 0}^{p-1} \max \left\{ V_{\lfloor\frac{\ell}{q} \rfloor} (K), V_{\lfloor\frac{p+q-1-\ell}{q} \rfloor} (K)\right\} \geq pV_0(K) - \frac{p^2-1}{4}.
$$
Combining the above inequality, together with the fact that $n' \leq n$, Equation~\eqref{eq:key}, and Lemma~\ref{lem:1}, we get 
$$p\lambda(L(p,p-1)) \geq pV_0(K) - \frac{p^2-1}{4}. $$
Note that by Proposition~\ref{proposition:cassonwalkerrecursive}, $\lambda(L(p,p-1))= \frac{p^2+2}{12p}-\frac{1}{4}$. Hence we get
$$\frac{p}{3} > \frac{4p^2-3p-1}{12p} \geq V_0(K)$$
which concludes the first part of the proof.

Lastly, if $q \equiv -1 \pmod{p}$, Equation~\eqref{eq:key} and Lemma~\ref{lem:1} imply that $$0\geq \sum_{\ell = 0}^{p-1} \max \left\{ V_{\lfloor\frac{\ell}{q} \rfloor} (K), V_{\lfloor\frac{p+q-1-\ell}{q} \rfloor} (K)\right\}.$$
The conclusion follows from Proposition~\ref{proposition:whennu+iszero} and the fact that $\{V_i(K)\}_{i\geq 0}$ is a sequence of non-negative integers.
\end{proof}

A $\Q$-homology sphere $Y$ is called an \emph{$L$-space} if $\rk \widehat{HF}(Y) = |H_1(Y; \Z)|$. It is known that all lens spaces are $L$-spaces. Examples of $L$-spaces which are not lens spaces include the Poincar\'e homology sphere, and branched double covers over alternating, non $2$-bridge knots. Let $\mathcal{L}'$ be the subgroup generated by $L$-spaces in $\qhom$. It is natural to ask whether the quotient $\mathcal{L}'/ \mathcal{L}$ is infinitely generated. It follows from~\cite{Aceto-Larson:2017-1} (see Corollary~\ref{cor:AK}), that the Poincar\'e homology sphere represents a non-trivial element in this quotient. Using Theorem~\ref{thm:B}, we exhibit infinitely many examples of irreducible $L$-spaces, each having infinite order in $\mathcal{L}'/ \mathcal{L}$.

\begin{corollary}\label{cor:L'/L} For any prime $p$, $S^3_{p/p-1}(T_{2,3})$ is an irreducible $L$-space that has infinite order in $\mathcal{L}' / \mathcal{L}$.
\end{corollary}
\begin{proof} 
If $p/q \geq 1$, $p/q$ Dehn surgery along $T_{2,3}$ is an $L$-space \cite{Ozsvath-Szabo:2005-1,Ozsvath-Szabo:2011-1}. Moreover, if follows from \cite{Moser:1971-1} that each $S^3_{p/p-1}(T_{2,3})$ is irreducible. The conclusion follows from Theorem~\ref{thm:B} and the fact that $V_0(T_{2,3})=1$ (see \cite{Ozsvath-Szabo:2003-1,Peters:2010-1}).
\end{proof}

Given a knot $K \subset S^3$, denote with $K_{p,q}$ the $(p,q)$-cable of $K$. In the next corollary we give some restrictions on integer surgeries on $K_{p,q}$ to be contained in $\mathcal{L}$.

\begin{corollary}\label{cor:cablev0} Let $q$ be prime and $p>1$. If $S^3_{pq}(K_{p,q})$ has finite order in $\Theta_\mathbb{Q}^3 / \mathcal{L}$, then $$0 \leq 3V_0(K) < q.$$ Furthermore, if $p \equiv -1 \pmod{q}$, then $V_0(K)=0$, or equivalently $\nu^+(K)=0$.
\end{corollary}

\begin{proof} From \cite{Gordon:1983-1}, we have $S^3_{pq}(K_{p,q}) \cong L(p,p-q)\# S^3_{q/p}(K)$. Hence $S^3_{pq}(K_{p,q})$ has finite order in $\Theta_\mathbb{Q}^3 / \mathcal{L}$ if and only if $S^3_{q/p}(K)$ has finite order in $\Theta_\mathbb{Q}^3 / \mathcal{L}$. Then the conclusion follows from Theorem~\ref{thm:B}.
\end{proof}

If we restrict to $(p,1)$-cable of knots, we get a stronger conclusion.

\begin{corollary}\label{cor:cable}
The $p$-surgery on the $(p,1)$-cable of a knot $K$ belongs to $\mathcal{L}$ if and only if the $\Z$-homology sphere $S^3_{1/p}(K)$ bounds a $\Q$-homology ball.
\end{corollary}
\begin{proof}

As before $S^3_{p}(K_{p,1}) \cong L(p,p-1)\# S^3_{1/p}(K)$, and $S^3_{p}(K_{p,1})$ is contained in $\mathcal{L}$ if and only if $S^3_{1/p}(K)$ is contained in $\mathcal{L}$. Then the conclusion follows from Corollary~\ref{cor:AK}.
\end{proof}

We recall one more lemma before the proof of Theorem~\ref{thm:C}. 
\begin{lemma}[{\cite[Lemma~4.5]{Aceto-Golla:2017-1}}]\label{lemma:inequalond}If $p>q>0$ are relatively prime, then \[\pushQED{\qed}4|d(L(p,q),\ell)| \leq p-1 \text{ for any } \ell \in \Z_p. \qedhere \] \end{lemma}

\begin{proof}[Proof of Theorem~\ref{thm:C}]

We start with the left-most inequality; if $S^3_p (K) \in \lens$, then by Theorem \ref{thm:main}, $S^3_p (K)\# L(p,q)$ bounds a $\Q$-homology ball for some $p> q>0$. Using an argument similar to the one given in the proof of Theorem~\ref{thm:B}, we can find a spin$^c$ structure on the $\Q$-homology ball that restricts to a spin$^c$ structure on $S^3_p (K)\# L(p,q)$ that corresponds to $(0, \ell) \in \Z_p \oplus \Z_p$, for some $\ell$. Hence using Proposition~\ref{proposition:Viformula} we get
$$d(S^3_p (K)\# L(p,q),(0,\ell)) = d(S^3_p (U), 0) - 2V_0(K) + d(L(p,q),\ell)=0.$$ Now, we get the desired inequality by applying Lemma~\ref{lemma:inequalond}.

Now we prove the right inequality.  Assume by contradiction that $ g_4(K) +1 \le \frac{p}{4}$; by Theorem~\ref{thm:main}, $S^3_p(K) \# L(p,q)$ is trivial in $\qhom$, for some $p> q>0$. Using the same line of reasoning as in the proof of Theorem~\ref{thm:B}, this implies that 
 \begin{equation}\label{eqn:sumofcorrection}
 \sum_{\ell = 0}^{p-1} \left( d(S^3_p(K),\ell) + d(L(p,q),\ell)\right) = 0,
 \end{equation}
hence 
 \begin{equation}\label{eqn:sumofcorrection2}
 p\lambda(L(p,p-1)) +p\lambda(L(p,q)) = 2(V_0(K) + 2V_1 (K) + \ldots + 2V_{\frac{p-1}{2}}(K)).
 \end{equation}
By~\cite[Theorem~2.3]{Rasmussen:2004-1}, the right-hand side is less than or equal to $g_4(K) (g_4(K) +1)$.

Since we are assuming $ g_4(K) +1 \le \frac{p}{4}$, Equation~\eqref{eqn:sumofcorrection2} implies that $\lambda(L(p,q)) - \lambda(L(p,1)) \le \frac{1}{4}\left( \frac{p}{4}-1\right)$. By Proposition~\ref{proposition:cassonwalker}, we see that $q \in \{1,2,3\}$. We examine these three cases separately.

If $q = 1$, Equation~\eqref{eqn:sumofcorrection2} implies that $\nu^+(K) = 0$, which contradicts our hypothesis. Using the fact from \cite[Corollary~3]{Kim-Livingston:2014-1} that the metabolizer surjects onto each cyclic summand of $H_1 (S^3_p(K) \# L(p,q);\Z)$, we obtain $$\{ d(S^3_p(K),\s) \}_{\s \in \text{Spin}^c(S^3_p(K))} = \{- d(L(p,q),\s) \}_{\s \in \text{Spin}^c(L(p,q))} .$$
It then follows from the proof of Proposition~2.5 in~\cite{Rasmussen:2004-1}, that if $q=2$ then $p \le 7$, and  if $q = 3$ then $p \le 13$. If $q = 2$ we obtain a contradiction, since $$8 \leq 4\nu^+(K) + 4 \leq 4g_4(K)+4 \leq p.$$ Finally, if $q = 3$, we see that $1\le g_4(K)\le2$. Then, by Equation~\eqref{eqn:sumofcorrection2} we have $$\lambda(L(p,p-1)) + \lambda(L(p,3)) \le\frac{g_4(K)(g_4(K) +1)}{p}.$$ Using Proposition~\ref{proposition:cassonwalkerrecursive}, the left-hand side can be written as 
$\frac{3p^2 -4}{36p} - \lambda(L(3,p)).$ Since $ |\lambda(L(3,p))| = \frac{1}{18}$, we obtain
\begin{equation}\label{eqn:gstarandp}
 \frac{p^2 - p -2 }{18} \le g_4(K)(g_4(K) +1).
\end{equation}
It is easy to see that Equation~\eqref{eqn:gstarandp} has no solutions when $1\le g_4(K)\le2$.
\end{proof}

Theorem~\ref{thm:C} implies that a knot $K$ such that $g_4(K) = V_0(K)$,  admits at most $2$ prime surgeries in $\lens$. In general, all such primes must be contained in an interval of length $4(g_4(K) - V_0(K)) +2$. It would be interesting to see whether similar results hold for non-prime surgery coefficients.

\begin{corollary}\label{cor:gstarobstruction}
Let $p$ be a prime and $K$ an $L$-space knot. If $p \ge 4g_4(K) + 4$, then $S^3_p (K) \in \mathcal{L}^\prime \setminus \lens$.
\end{corollary}
\begin{proof}
The result follows from Theorem \ref{thm:C} and the fact that for an $L$-space knot, every surgery coefficient greater than $2g_4(K) -1$ is an $L$-space \cite{Ozsvath-Szabo:2005-1,Ozsvath-Szabo:2011-1}.
\end{proof}

\bibliographystyle{alpha}
\def\MR#1{}
\bibliography{bib}
\end{document}